\documentclass{amsart}[12pt]
\usepackage{amsmath, amsfonts, amssymb, amsthm, euscript, amscd, latexsym, mathrsfs, bm, color}
 \usepackage{titlesec}

  \usepackage{hyperref}

\def\aa#1{ \begin{align*} #1 \end{align*} }
\def\aaa#1{ \begin{align} #1 \end{align} }
\def\mm#1{ \begin{multline*} #1 \end{multline*} }
\def\mmm#1{ \begin{multline} #1 \end{multline} }

\newtheorem{thm}{\sc Theorem}
\newtheorem{lem}{\sc Lemma}

\newtheorem{pro}{\sc Proposition}
\newtheorem{rem}{\sc Remark}

\newcommand{\pl}{\partial}
\newcommand{\gt}{\geqslant}
\newcommand{\lt}{\leqslant}
\newcommand{\te}{\theta}

\newcommand{\al}{\alpha}
\newcommand{\gm}{\gamma}
 \newcommand{\Gm}{\Gamma}
 \newcommand{\Dl}{\Delta}
 \newcommand{\la}{\lambda}
 \newcommand{\La}{\Lambda}
 \newcommand{\sg}{\sigma}

\newcommand{\mc}{\mathcal}

\newcommand{\C}{{\rm C}}

\newcommand{\td}{\tilde}

\newcommand{\E}{\mathbb E}

\newcommand{\x}{\times}

\newcommand{\PP}{\mathbb P}

\newcommand{\rf}{\eqref}
\newcommand{\bi}{\begin{itemize}}
\newcommand{\ei}{\end{itemize}}

\newcommand{\lb}{\label}
\newcommand{\fdot}{\,\cdot\,}

\def\Rnu{{\mathbb R}}

\def\ffi{\varphi}

\newcommand{\sig}{\varsigma}

\def\begeq{\begin{equation} \begin{cases}} 
\def\endeq{ \end{cases} \end{equation}}
\def\eq#1{ \begeq #1 \endeq }
\def\bege{\begin{equation*} \begin{cases}} 
\def\ende{ \end{cases} \end{equation*}}
\def\eqq#1{ \bege #1 \ende}
\def\eqn#1{ \begin{equation} \begin{aligned} #1 \end{aligned}\end{equation}}

\def\com#1{}

\long\def\symbolfootnote[#1]#2{\begingroup%
\def\thefootnote{\fnsymbol{footnote}}\footnote[#1]{#2}\endgroup}
\titleformat{\section}[hang]{\large\bfseries}{\thesection.}{1ex}{}{}
\titleformat{\subsection}[hang]{\normalsize\bfseries}{\thesubsection}{2ex}{}{}
\titleformat{\subsubsection}[hang]{\small\bfseries}{\thesubsubsection}{2ex}{}{}

\include{srctex.sty}

\title[Gaussian density estimates for solutions of fully coupled FBSDEs]{Gaussian density estimates for solutions \\ 
of fully coupled forward-backward SDEs}

\author{Christian Olivera}
\address{Departamento de Matem\'atica, Universidade Estadual de Campinas, Campinas, Brazil}
\email{colivera@ime.unicamp.br}

\author{Evelina Shamarova}
\address{Departamento de Matem\'atica, Universidade Federal da Para\'iba, Jo\~ao Pessoa, Brazil}
\email{evelina@mat.ufpb.br}

\begin{document}

\maketitle

\begin{abstract} 
We obtain upper and lower Gaussian density estimates for the laws of each component of the solution 
 to a one-dimensional fully coupled forward-backward SDE. 
 Our approach relies on the link between FBSDEs and quasilinear parabolic PDEs,
 and is fully based on the use of classical results on PDEs rather than on manipulation of FBSDEs, 
 compared to other papers on this topic.
 This essentially simplifies the analysis. 
\end{abstract}

\section{Introduction}
 Forward-backward stochastic differential equations (FBSDEs)
have numerous applications in stochastic control theory and mathematical finance
(see, for instance, \cite{Ci}, \cite{Karo}, \cite{Par},  and \cite{Yong}).
Several recent papers \cite{aboura,antonelli, mastrolia} studied existence of densities and density estimates for the laws of solutions of one-dimensional backward SDEs (BSDEs) (\cite{aboura, antonelli, mastrolia}).  To the best of authors' knowledge, the aforementioned problem has never been addressed in connection to the laws of solutions to fully coupled FBSDEs.

  In this paper, we are concerned with the fully coupled one-dimensional FBSDE
\eq{
\lb{fbsde}
X_t = x+ \int_0^t f(s,X_s, Y_s,Z_s) ds + \int_0^t \sg(s,X_s,Y_s) dB_s,\\
Y_t = h(X_T) + \int_t^T g(s,X_s,Y_s,Z_s) ds - \int_t^T Z_s dB_s,
}
where $B_t$ is a one-dimensional standard Brownian  motion, and $f$, $\sg$, $g$, and $h$ are functions 
defined on appropriate spaces and taking values in $\Rnu$. It
is known (see, e.g., \cite{ma}, \cite{Par},  \cite{Yong}, and also references therein) that there exists
a unique  $\mc F_t$-adapted solution $(X_t,Y_t, Z_t)$ of system (\ref{fbsde})  under some appropriate smoothness and boundedness 
conditions on its coefficients, where $\mc F_t$ is the augmented filtration generated by the Brownian motion $B_t$.

Our goal is to provide conditions that guarantee the existence of the densities for the laws of $X_t$, $Y_t$, and $Z_t$,
 and that allow Gaussian estimates of these densities. Additionally,
we obtain estimates of the tail probabilities of the laws of the solution components.
Our approach works for a large class of the FBSDE coefficients. 
In particular, the BSDE generator $g$
is not assumed to depend just on some of the spatial variables (unlike \cite{antonelli, mastrolia}),
or to be linear in $Z_s$ (unlike \cite{aboura}). To add even more generality, we obtain our density estimates
in the situation when the generator $g$ has the quadratic growth in the last variable, and hence,
it is not necessarily Lipschitz in $Z_s$.
Our method relies on the analysis of the quasilinear parabolic PDE associated to
FBSDE \rf{fbsde}:
\eq{
\lb{final}
\frac12 \sg^2(t,x,u) \pl^2_{xx} u + f(t,x,u, \sg(t,x,u)\pl_x u)\pl_x u + g(t,x,u, \sg(t,x,u)\pl_x u) \\ \;\; + \pl_t u = 0,
\hspace{2cm} u(T,x) = h(x),
} 
where $u$, $\pl_xu$, and $\pl^2_{xx}u$ are everywhere evaluated at $(t,x)$.
It is well known (see, e.g., \cite{ma}) that if $u$ is the $\C^{1,2}_b$-solution to final value problem \rf{final},
then it is related to the solution of FBSDE \rf{fbsde} by the formulas
\aaa{
\lb{link1}
Y_t = u(t,X_t), \quad Z_t = \pl_xu(t,X_t) \sg(t,X_t,u(t,X_t,)),
}
where $X_t$ is the unique $\mc F_t$-adapted solution to the SDE
\aaa{
\lb{sde}
& X_t = x + \int_0^t \td f(t,X_t) + \int_0^t \td \sg(t,X_t) dB_t  \qquad \text{with}\\
\lb{coeff-sde}
&\td f(t,x) = f(t,x,u(t,x), \pl_xu(t,x)\sg(t,x,u(t,x))), \quad  \td\sg(t,x) = \sg(t,x,u(t,x)).
}
Since the Malliavin differentiability and the existence of bounds for $D_rX_t$ are well known facts
(see, e.g., \cite{nualart}, \cite{detemple}), 
then, provided that the coefficients of PDE \rf{final} are sufficiently smooth, 
the Malliavin differentiability of $Y_t$ and $Z_t$ follows immediately, and, moreover, the existence
of bounds for $D_rY_t$ and $D_rZ_t$ is reduced to the existence of positive lower bounds for 
$\pl_x u$ and $\pl_x\big(\pl_x u\, \sg\big)$. This can be done by the classical comparison theorem 
for PDEs (see, e.g., \cite{friedman}). 

Let us remark that our assumptions allow the BSDE generators $g(t,x,u,p)$ to have the quadratic growth in $p$. 
It happens because the four step scheme, developed in \cite{ma}, 
also works for a quadratic BSDE, provided that it is one-dimensional.
This follows from the version of the existence and uniqueness theorem, obtained in \cite{lady},
 for the associated one-dimensional 
PDE \rf{final}. Remark that density (tail probability) estimates
for the law of the $Z_s$-component of quadratic BSDEs are important for some numerical schemes, 
as it was mentioned in \cite{mastrolia}.

In comparison to our approach, papers \cite{aboura, antonelli, mastrolia} mainly use manipulations of the BSDE itself, 
such as, consideration of  the BSDEs for the second order Malliavin derivatives of the solution processes, 
Girsanov's transformation,
It\^o's formula for various functions of the solution, etc, to arrive at the existence of
estimates for the Malliavin derivatives $D_rY_t$ and $D_rZ_t$. 

Overall,  compared to the previous works, our analysis is simpler, many of the assumptions are dropped  or 
easier formulated (c.f. \cite{mastrolia}), while the FBSDE itself is, overall, more general (in particular, fully coupled)
and the density estimates hold on the entire real line.


\section{Preliminaries}

For simplicity, all PDEs considered in this section are one-dimensional and with respect to one space variable,
although all the results are valid for PDEs of several space variables.
\subsection{Useful function spaces}
We start by defining some function spaces used in this paper.

The H\"older space $\C^{2+\beta}_b(\Rnu)$,
 $\beta\in (0,1)$,
is understood as the (Banach) space with the norm
\aa{
\|\phi\|_{\C^{2 +\beta}_b(\Rnu)} = \|\phi\|_{\C^2_b(\Rnu)} + [\phi'']_\beta,
\quad \text{where} \quad 
[\td \phi]_\beta = \sup_{x,y\in\Rnu, \,  0<|x-y|<1}\frac{|\td\phi(x)- \td\phi(y)|}{|x-y|^\beta},
}
and $\C^2_b(\Rnu)$ denotes the space of twice continuously differentiable functions  on $\Rnu$ with bounded derivatives up to the second order.

For a function $\phi(x,\xi)$ of more than one variable, the H\"older constant 
with respect to $x$ is defined as 
\aa{
[\phi]^x_{\beta} =  \sup_{x,x'\in\Rnu, \, 0<|x-x'|<1}  \frac{|\phi(x,\xi) - \phi(x',\xi)|}{|x-x'|^\beta},
}
i.e., it is understood as a function of $\xi$.

The H\"older spaces $\C^{1+\frac{\beta}2,2+\beta}_b([0,T]\x \Rnu)$, $\C^{\frac{\beta}2,\beta}_b([0,T]\x\Rnu)$,
$\C^{\frac{\beta}2,1+\beta}_b([0,T]\x\Rnu)$, and $\C^{0,\beta}_b([0,T]\x\Rnu)$ 
 ($\beta\in (0,1)$) are 
defined, respectively, as Banach spaces of functions $\phi(t,x)$ possessing the finite norms
\aa{
&\|\phi\|_{\C_b^{1+\frac{\beta}2,2+\beta}([0,T]\x \Rnu)} = 
\|\phi\|_{\C_b^{1,2}([0,T]\x \Rnu)} + 
\sup_{t\in [0,T]}[\pl_t \phi]_{\beta}^x + \sup_{t\in [0,T]}[\pl^2_{xx} \phi]_{\beta}^x \\
& \hspace{4.2cm}+  \sup_{x\in \Rnu}[\pl_t \phi]_{\frac{\beta}2}^t + \sup_{x\in \Rnu}[\pl_x \phi]_{\frac{1+\beta}2}^t
+ \sup_{x\in \Rnu}[\pl^2_{xx} \phi]_{\frac{\beta}2}^t; \\
&\|\phi\|_{\C_b^{\frac{\beta}2,\beta}([0,T]\x \Rnu)} = 
\|\phi\|_{\C_b([0,T]\x \Rnu)} + \sup_{t\in [0,T]}[\phi]_{\beta}^x + \sup_{x\in \Rnu}[\phi]_{\frac{\beta}2}^t;\\
&\|\phi\|_{\C_b^{\frac{\beta}2,1+\beta}([0,T]\x \Rnu)} = 
\|\phi\|_{\C_b([0,T]\x \Rnu)}  + \|\pl_x \phi\|_{\C_b^{\frac{\beta}2,\beta}([0,T]\x \Rnu)}; \\
&\|\phi\|_{\C_b^{0,\beta}([0,T]\x \Rnu)} =  \|\phi\|_{\C_b([0,T]\x \Rnu)} 
+ \sup_{t\in [0,T]}[\phi]_{\beta}^x,
}
where
$\C_b^{1,2}([0,T]\x \Rnu)$ is the space of bounded continuous functions
whose derivatives up to the first order in $t\in [0,T]$ and the second order in $x\in\Rnu$  
are bounded and continuous on $[0,T]\x \Rnu$, and $\C_b([0,T]\x \Rnu)$ is the space
of bounded continuous functions.
%
\subsection{Some results on quasilinear parabolic PDEs}
Here we formulate some results on linear and quasilinear parabolic PDEs which will be useful in the next section.

Consider the Cauchy problem for a one-dimensional PDE of one space variable
\eq{
\lb{pde}
a(t,x,u) \pl^2_{xx} u + f(t,x,u,\pl_x u) \pl_x u + g(t,x,u, \pl_x u) - \pl_t u = 0, \\
u(0,x) = h(x),
}
where $u$, $\pl_x u$, $\pl_t u$, and $\pl^2_{xx} u$ are everywhere evaluated at $(t,x)$.

In what follows, $(t,x,u,p)$ denotes the element of $[0,T]\x \Rnu \x \Rnu \x \Rnu$, 
$f$ and $g$ are functions of $(t,x,u,p)$, and $a$ is a function of $(t,x,u)$. Further,
$\pl_t$, $\pl_x$, $\pl_u$, and $\pl_p$ denote the partial derivatives w.r.t. $t$, $x$, $u$, and $p$,
respectively. 

The theorem below, proved in  \cite{lady} (Theorem 8.1, Section V, p. 495),  provides
the existence and uniqueness of solution to problem  \rf{pde}.
\begin{thm}
\lb{lady}
Assume conditions (i)--(vii) below:
\bi
\item[(i)] for all $(t,x,u) \in [0,T]\x \Rnu \x \Rnu$, $\nu(|u|) \lt a(t,x,u) \lt \mu(|u|)$,
where $\nu$ and $\mu$ are non-increasing and, respectively, non-decreasing positive functions;
\item[(ii)] for all $(t,x,u,p) \in [0,T]\x \Rnu \x \Rnu \x \Rnu$, $g(t,x,u,p) u \lt c_1 + c_2 |u|^2$,
where $c_1$ and $c_2$ are positive constants;
\item[(iii)] the function $h$ is of class $\C^{2+\beta}_b(\Rnu)$, $\beta\in (0,1)$.
\item[(iv)]  $\pl_x a$ and $\pl_u a$ exist and
$|a| + |\pl_x a| + |\pl_u a| \lt \al$, where $\al>0$ is a constant;
\item[(v)] there exists a positive non-decreasing function $\td \mu$ such that
$|f| \lt \td\mu(|u|)(1+|p|)$ and $|g| \lt \td\mu(|u|)(1+|p|^2)$
everywhere on $[0,T]\x\Rnu \x\Rnu \x\Rnu$;
\item[(vi)] the functions $a$, $\pl_x a$, $\pl_u a$,  $f$, and $g$
are H\"older continuous in $t$, $x$, $u$, and $p$ with exponents $\frac{\beta}2$, $\beta$,
$\beta$, and $\beta$, respectively, and globally bounded H\"older constants;
\item[(vii)] the derivatives $\pl_u f$, $\pl_u g$, $\pl_p f$, $\pl_p g$ exist and
$
\sup_{\substack{
(t,x) \in [0,T]\x \Rnu\\
|u|+|p| \lt N }}
\big(|\pl_u f| + |\pl_u g| + |\pl_p f| + |\pl_p g|\big) \lt \gm(N),
$
where $\gm(N)$ is a positive constant depending on $N$.
\ei
Then, there exists a unique $\C^{1+\frac{\beta}2, 2+\beta}_b([0,T]\x\Rnu)$-solution to problem \rf{pde}.
\end{thm}
Now consider a Cauchy problem for a linear PDE:
\eq{
\lb{pde-lin}
a(t,x) \pl^2_{xx} u + b(t,x) \pl_x u + c(t,x) u + g(t,x) - \pl_t u = 0, \\
u(0,x) = \ffi(x).
}
We have the following result, proved in \cite{friedman} (Theorem 12, p. 25 and Theorem 10, p. 44), 
on the solvability of problem \rf{pde-lin} and the representation of its solution via
the fundamental solution $\Gm(t,x,s,z)$ to PDE  \rf{pde-lin}.
\begin{thm}
\lb{fundamental}
Let PDE \rf{pde-lin} be uniformly parabolic, and let
the coefficient $a$  of \rf{pde-lin} be of class $\C^{\frac{\beta}2,\beta}_b([0,T]\x \Rnu)$, $\beta\in (0,1)$.
Further let the coefficients $b$, $c$, and the function $g$ be of class
$\C^{0,\beta}_b([0,T]\x \Rnu)$, and the initial condition $\ffi$
be of class $\C_b(\Rnu)$.  Then, there exists a unique $\C^{1,2}_b([0,T]\x \Rnu)$-solution to problem \rf{pde-lin}.
Moreover, this solution takes the form
\aa{
 u(t,x) = \int_{\Rnu}\Gm(t,x,0,z) \ffi(z) dz - \int_0^t \int_{\Rnu} \Gm(t,x,s,z) g(s,z) ds dz.
 } 
\end{thm}
Introduce the linear differential operator
\aa{
L u = a(t,x) \pl^2_{xx} u + b(t,x) \pl_x u + c(t,x) u - \pl_t u.
}
Theorem \ref{hold-unique} below, provides conditions when the solution $u$ to \rf{pde-lin}
belongs to class $\C^{1+\frac{\beta}2, 2+\beta}_b([0,T]\x\Rnu)$. The theorem was
proved in \cite{lady} (Theorem 5.1, p. 320).
\begin{thm}
\lb{hold-unique}
Let PDE \rf{pde-lin} be uniformly parabolic, and the coefficients of the operator $L$ and
the function $g$ belong to class $\C^{\frac{\beta}2,\beta}_b([0,T]\x \Rnu)$. Further let
the initial condition $\ffi$ belong to class $\C^{2+\beta}_b(\Rnu)$. Then, problem \rf{pde-lin}
has a unique  $\C^{1+\frac{\beta}2, 2+\beta}_b([0,T]\x\Rnu)$-solution $u(t,x)$. 
\end{thm}
The following below comparison theorem, proved in \cite{friedman} (Theorem 9, p 43), will be an important tool
in the next section.
\begin{thm}
\lb{comparison}
Let the coefficients of $L$ be bounded and continuous on $[0,T]\x\Rnu$.
Assume $Lu\lt 0$ on $(0,T] \x \Rnu$ and $u$ is bounded.
If $\ffi(x) \gt 0$ on $\Rnu$, then $u(t,x) \gt 0$ on $[0,T]\x \Rnu$.
\end{thm}
\subsection{A link between FBSDEs and quasilinear parabolic PDEs}
It is well known that there is a link between FBSDE \rf{fbsde} and a quasilinear parabolic PDE of form \rf{pde}
(see, e.g., \cite{ma}). Specifically, the final value problem
for the PDE associated to FBSDE \rf{fbsde} takes form \rf{final}.
By introducing the time-changed function $\te(t,x) = u(T-t,x)$, we transform \rf{final} to the Cauchy  
problem
\mmm{
\lb{cauchy}
\frac12 \sg^2(T-t,x,\te) \pl^2_{xx} \te + f(T-t,x,\te, \sg(T-t,x,\te)\pl_x \te)\pl_x \te \\ + g(T-t,x,\te, \sg(T-t,x,\te)\pl_x \te) - \pl_t \te = 0,\quad
\te(0,x) = h(x).
} 
Remark that under assumptions (i)--(vii) of Theorem \ref{lady}, the existence
and uniqueness of a $\C^{1+\frac{\beta}2,2+\beta}_b([0,T]\x \Rnu)$-solution to \rf{cauchy}
is established, and is equivalent to the existence and uniqueness of a $\C^{1+\frac{\beta}2,2+\beta}_b([0,T]\x \Rnu)$-solution $u$ to final value problem \rf{final}.
The theorem below provides an explicit solution to FBSDE \rf{fbsde} via the solution $u$.
\begin{thm}
\lb{link} Let the functions $f$, $g$, and $h$ satisfy assumptions (ii), (iii), (v)--(vii) of Theorem \ref{lady}.
Further let the function $\sg$ satisfy assumptions (i), (iv), and (vi) of the same theorem in the place of the function $a$.
Then, there exists a unique $\mc F_t$-adapted solution $(X_t,Y_t,Z_t)$
to FBSDE \rf{fbsde}. Moreover, this solution takes form \rf{link1} with $u$ being
the unique $\C^{1+\frac{\beta}2,2+\beta}_b([0,T]\x \Rnu)$-solution, $\beta\in (0,1)$, to 
problem \rf{final}.
\end{thm}
\begin{rem}
\rm The solution to FBSDE \rf{fbsde} is understood as in \cite{ma}.
\end{rem}
The proof of Theorem \ref{link} is exactly the same as the proof of Theorem 4.1 in \cite{ma},
where the latter result is known as the four step scheme. It relies exceptionally  
on the existence of the unique $\C^{1,2}_b$-solution to Cauchy problem \rf{cauchy}.
This implies that the assumptions of Theorem \ref{lady} guarantee the existence of a unique 
solution to FBSDE \rf{fbsde}. These assumptions turn out to be more general than in \cite{ma},
but they are restricted to the case of just one PDE.
Remark, that the Cauchy problem for systems of PDEs was not actually solved in \cite{lady},
so the authors of \cite{ma} had to fill this gap imposing own assumptions. However, for the case
of just one PDE, the Cauchy problem is solved in \cite{lady}, and the result is represented 
by Theorem 8.1 in Section V (p. 495), so we make use of its more general assumptions.

\subsection{The Malliavin derivative}

Here we describe the elements from the Malliavin calculus  that we  need in the paper.  
We refer the reader to \cite{nualart} for a more complete exposition. 

Consider ${\mathcal{H}}$ a real separable Hilbert space and $(B (\varphi), \varphi\in{\mathcal{H}})$ an isonormal Gaussian process \index{Gaussian process} on a probability space $(\Omega, A, P)$, which is a centered Gaussian family of random variables such that 
$\mathbb{E}\left( B(\varphi) B(\psi) \right)  = \langle\varphi, \psi\rangle_{{\mathcal{H}}}$.

We denote by $D$  the Malliavin  derivative operator that acts on smooth functions of the form $F=g(B(\varphi _{1}), \ldots , B(\varphi_{n}))$ ($g$ is a smooth function with compact support and $\varphi_{i} \in {\mathcal{H}}, i=1,...,n$):
\begin{equation*}
DF=\sum_{i=1}^{n}\frac{\partial g}{\partial x_{i}}(B(\varphi _{1}), \ldots , B(\varphi_{n}))\varphi_{i}.
\end{equation*}
It can be checked that the operator $D$ is closable from $\mathcal{S}$ (the space of smooth functionals as above) into $ L^{2}(\Omega; \mathcal{H})$ and it can be extended to the space $\mathbb{D}^{1,p}$ which is the closure of $\mathcal{S}$ with respect to the norm
\begin{equation*}
\Vert F\Vert _{1,p}^{p} = \mathbb{E} |F|^{p} + \mathbb{E} \Vert DF\Vert _{\mathcal{H}} ^{p}. 
\end{equation*}
In our paper, $\mathcal{H} = L_2([0,T])$ and $B(\ffi) = \int_0^T \ffi(t) dB_t$. 

\subsection{Gaussian density estimates} 
Theorem \ref{privault-dung} below is an important tool that we will use to obtain the existence of densities
and density estimates. It  was proved in \cite{privault} (Theorem 2.4).
\begin{thm}
\lb{privault-dung}
 Let $F \in  D^{1,2}$ be a random variable such that  
 \aaa{
 \lb{condition}
 0 < l \lt  \int_0^\infty D_sF \, \E[D_s F |\mc F_s]ds \lt  L \quad \text{a.s.},
}
where $l$ and $L$  are constants. Then, 
$F$ possesses a density $p_F$ with respect to the Lebesgue measure. Moreover,
for almost all $x\in\Rnu$, the density $p_F$ satisfies 
\aa{     
\frac{\E |F-\E[F]|}{2L} \exp\Big(-\frac{\big(x-\E[F]\big)^2}{2l}\Big) \lt 
p_F(x) \lt \frac{\E |F-\E[F]|}{2l} \exp\Big(-\frac{\big(x-\E[F]\big)^2}{2L}\Big).
}
Furthermore, for all $x>0$, the tail probabilities satisfy
\aa{     
\PP(F \gt x) \lt \exp\Big(-\frac{\big(x-\E[F]\big)^2}{2L}\Big) 
\quad \text{and} \quad  \PP(F \lt -x) \lt \exp\Big(-\frac{\big(x+\E[F]\big)^2}{2L}\Big). 
}
\end{thm}
\begin{rem}
\rm Theorem \ref{privault-dung} was, in fact, obtained in \cite{privault} for centered random variables $F$.
However, since  $p_F(x) = p_{F-\E[F]}(x-\E[F])$, where $p_{F-\E[F]}$ 
is the density function for $F-\E[F]$, and condition \rf{condition} does not change if we replace $F$ with $F-\E[F]$,
the statement of Theorem \ref{privault-dung} follows immediately.
\end{rem}

\subsection{Malliavin derivatives of solutions to SDEs}
Consider SDE \rf{sde}, where the coefficients are given by \rf{coeff-sde} and $u(t,x)$ is the unique
$\C^{1+\frac{\beta}2,2+\beta}_b([0,T]\x \Rnu)$-solution to problem \rf{final}. It is known that
(see, e.g., \cite{nualart}) if the  coefficients  of an SDE are differentiable with
bounded derivatives, its solution is Malliavin differentiable. 
It is also known that if, additionally, $\td\sg$ 
is bounded away from zero, then,
 by means of Lamperti's transform $\eta(t,x) = \int_0^x \frac1{\td\sg(t,\xi)} d\xi$ (\cite{lamperti}),
the Malliavin derivative of $X_t$ can be explicitly computed. 
The algorithm is well known (see, e.g., \cite{detemple}), so we skip the computation, and write the final result:
\aaa{
\lb{dx}
D_r X_t = \td\sg(t,X_t) e^{\int_r^t \psi(s,X_s) ds},
}
where $\psi(s,x) = \frac{2\td f(s,x)\pl_x\td\sg(s,x)}{\td\sg^2(s,x)}-\frac{\pl_x\td f(s,x) + 
\td f(s,x) \pl^2_{xx}\td \sg(s,x)+ \pl_s\td\sg(s,x)}{\td\sg(s,x)} - \frac12\pl^2_{xx}\td \sg(s,x)\td\sg(s,x)$.

\section{Results}

In this section, we prove that the laws of $X_t$, $Y_t$, and $Z_t$ possess densities 
with respect to the Lebesgue measure, and obtain
Gaussian estimates for the densities and tail probabilities of these laws.

In what follows, we will make use of assumptions (A1)--(A9) below. 
Assumptions (A1)--(A3) are required to obtain density estimates for the law  of $X_t$.
\bi
\item[\bf (A1)] For all $(t,x,u) \in [0,T]\x \Rnu \x \Rnu$, $\nu(|u|) \lt \sg(t,x,u) \lt \mu(|u|)$,
where $\nu$ and $\mu$ are non-increasing and, respectively, non-decreasing positive functions; 
\item[\bf (A2)] the functions $f$, $g$, and $h$ satisfy conditions 
(ii), (iii), and (v)--(vii) of Theorem \ref{lady};
\item[\bf (A3)] 
the derivatives $\pl_x \sg$, $\pl_u \sg$, 
exist and are 
H\"older continuous in $t$, $x$,  $u$ with exponents $\frac{\beta}2$, 
$\beta$, $\beta$, respectively, and globally bounded H\"older constants; further, $\pl_s\sg$ exists, and
$|\sg| +  |\pl_s \sg| + |\pl_x \sg| + |\pl_u \sg| \lt \al$ for some constant $\al>0$.
\ei
Assumptions (A4) and (A5) below should be added to (A1)--(A3) to obtain density estimates for the law of $Y_t$.
Remark that under (A1)--(A3), the solution $u$ to problem \rf{final} possesses a bound for  $|\pl_xu|$.
This bound will be denoted by $M_1$. Also, we recall that the bound for $|u|$ is denoted by $M$.
\bi
\item[\bf (A4)]
In the region $[0,T] \x \mc R$, where $\mc R =\Rnu \x \{|u|\lt M\}\x \{|p| \lt M_1\}$,
$\pl_x f$, $\pl_x g$, $\pl_u f$, $\pl_u g$, $\pl_p f$, $\pl_p g$ exist, are bounded
 and H\"older continuous in $t$, $x$, $u$, $p$ with exponents $\frac{\beta}2$, $\beta$,
$\beta$, $\beta$, respectively, and bounded H\"older constants;
  \item[\bf (A5)] 
either (a) or (b) holds: 
\aa{
&\text{(a)\,} \quad 
h'\gt 0 \quad \text{and} \; \inf_{(x,u,p)\in \mc R} \pl_x g(t,x,u,p) > 0 \; \text{for all} \; t\in (0,T];\\
&\text{(b)\,} \quad
h' \lt 0\quad \text{and} \;  \sup_{(x,u,p)\in \mc R} \pl_x g(t,x,u,p) < 0  \; \text{for all} \; t\in (0,T].
}
\ei
 Finally, to estimate the density of the law of $Z_t$, 
  assumption (A5) should be replaced with assumption (A5') below, and, additionally, (A6)--(A9) should be in force.
  \bi
\item[\textbf{(A5}\textrm{'}\textbf{)}] 
For all $(t,x,u,p) \in (0,T] \x\mc R$, $\pl_x g \gt 0, \; h'\gt 0$. 
\ei
Further, (A6)--(A9) read:
\bi
\item[\bf (A6)] $\pl_x \sg\gt 0$, $\pl_u \sg\gt 0$ on $[0,T]\x\Rnu \x \{|u|\lt M\}$;
\item[\bf (A7)]  $\pl^2_{px} f$, $\pl^2_{pu}f$, $\pl^2_{pp} f$, $\pl^2_{xx}f$,
$\pl^2_{xu} f$, $\pl^2_{uu} f$, $\pl^2_{px} g$, $\pl^2_{pu}g$, $\pl^2_{pp} g$, $\pl^2_{xx}g$,
$\pl^2_{xu} g$, $\pl^2_{uu} g$ exist on $[0,T]\x \mc R$, are bounded and H\"older continuous in $t$, $x$, $u$, $p$ with 
exponents $\frac{\beta}2$, $\beta$, $\beta$, $\beta$, respectively, and bounded H\"older constants;
\item[\bf (A8)] 
for all $t\in (0,T]$, 
$\inf_{(x,u,p)\in \mc R}\pl^2_{xx} g >0$ and $h''\gt 0$;
\item[\bf (A9)]  the following inequalities hold on
$[0,T]\x\mc R$
\eqq{
\pl^2_{xx} f + 2\pl^2_{xu} g + \pl_p g \pl^2_{xx} \sg + \pl^2_{px} g \, \pl_x \sg \gt 0, \\
\pl^2_{uu} g + 2\pl^2_{ux} f + 2\pl^2_{px}g \, \pl_u \sg + 2\pl^2_{pu}g\, \pl_x \sg + 2\pl^2_{px} f \, \pl_x \sg 
+ 2\pl_p g \pl^2_{xu} \sg +\pl^2_{pp} g(\pl_x\sg)^2\gt 0,\\
\pl^2_{uu} f  + 2\pl^2_{pu} g\, \pl_u\sg + 2\pl^2_{pu}f \pl_x \sg + 2\pl^2_{px} f \pl_u \sg 
+ 2\pl^2_{pp} g\, \pl_x\sg\pl_u\sg + 2\pl_p f \pl^2_{xu}\sg + \pl_p g \pl^2_{uu}\sg \\
\hspace{10cm} + \pl^2_{pp} f (\pl_x\sg)^2  \gt 0,\\
2\pl^2_{up} f \pl_u\sg + 2\pl^2_{pp}f \pl_x\sg \pl_u\sg + \pl_p f \pl^2_{uu}\sg + \pl^2_{pp} g (\pl_u\sg)^2 \gt 0, \\
\pl^2_{pp} f  \gt 0.
}
\ei

\subsection{Density estimates for the law of $X_t$}
\begin{thm}
Let (A1)--(A3) hold. Then, the law of $X_t$ has a density $p_{X_t}$ with respect to the Lebesgue measure.
Moreover, for almost all $x\in \Rnu$, $p_{X_t}$ satisfies the estimate
\aaa{  
  \lb{est-x} 
\frac{\E |X_t-\E[X_t]|}{2\Xi(t)} \exp\Big(-\frac{\big(x-\E[X_t]\big)^2}{2\xi(t)}\Big) \lt 
p_{X_t}(x) \lt \frac{\E |X_t-\E[X_t]|}{2\xi(t)} \exp\Big(-\frac{\big(x-\E[X_t]\big)^2}{2\Xi(t)}\Big),
}
where $\xi(t)$ and $\Xi(t)$ are positive functions that can be computed explicitly. 
Further, for all $x>0$, the tail probabilities of $X_t$ satisfy
\aaa{
\lb{tail-x} 
\PP(X_t > x) \lt \exp\Big(-\frac{\big(x-\E[X_t]\big)^2}{2\Xi(t)}\Big) \quad \text{and} \quad
\PP(X_t< -x) \lt \exp\Big(-\frac{\big(x + \E[X_t]\big)^2}{2\Xi(t)}\Big).
}
\end{thm}
\begin{proof}
Note that, under (A1)--(A3),  the solution $u$ to problem \rf{final} and its derivative $\pl_x u$, $\pl_s u$, and $\pl^2_{xx} u$ 
are bounded. Hence,
$\td f$, $\td\sg$, $\pl_x\td f$, $\pl_s\td\sg$, $\pl_x \td\sg$, and $\pl^2_{xx} \td\sg$ are bounded as well
(functions $\td f$ and $\td \sg$ are defined by \rf{coeff-sde}). 
Further, by (A3), on $[0,T]\x \Rnu$, $\td\sg(t,x)\gt\nu(M)$, where by $M$ is the bound for $|u|$.
Therefore,  the function $\psi$ in 
\rf{dx} is bounded. Let $M_\psi$ be its bound. Formula \rf{dx} allows us to estimate  $D_rX_t$ as follows
\aaa{
\lb{dx-est}
\nu(M) e^{-M_\psi t}\lt D_rX_t \lt \mu(M) e^{M_\psi t} \quad \text{a.s.}
} 
This implies that
\aa{
t \nu(M)^2 e^{-2M_\psi t} \lt \int_0^t D_r X_t \E[D_r X_t | \mc F_r] dr \lt t\mu(M)^2 e^{2M_\psi t}.
}
Remark that $D_r X_t = 0$ if $r>t$. By Theorem \ref{privault-dung}, the law of $X_t$ has a density with respect to 
the Lebesgue measure and estimate \rf{est-x} holds with $\xi(t) = t \nu(M)^2 e^{-2M_\psi t}$
and $\Xi(t) = t\mu(M)^2 e^{2M_\psi t}$. Moreover, the tail probabilities of $X_t$ satisfy \rf{tail-x}. 
\end{proof}
\subsection{Density estimates for the law of $Y_t$}
To estimate the density for $Y_t$, we will use the formula $Y_t = u(t,X_t)$, where $u$
is the unique $\C^{1+\frac{\beta}2,2+\beta}_b$-solution to problem \rf{final}. This formula immediately
implies that $Y_t$ is Malliavin differentiable and $D_rY_t = \pl_x u(t,X_t) D_r X_t$.

Below, we prove that under (A1)--(A5), there exists a positive function $m(t)$, $t\in [0,T]$,
such that 
\eqn{
\lb{D-bound}
\text{either} \;\;  &\inf_{x\in\Rnu} \pl_x u(t,x) \lt -m(t) \; &&\text{for all}  \; t\in [0,T] \\
\text{or} \quad
&\sup_{x\in\Rnu} \pl_xu(t,x) \gt m(t) \; &&\text{for all}  \; t\in [0,T]. 
}
To this end,  we obtain a PDE for the function $v= \pl_x u$. 
We start by considering linear PDE \rf{pde-lin} and prove that we can differentiate it with respect to $x$.
The following result can be viewed as a corollary of Theorem \ref{fundamental}.
\begin{pro}
\lb{pr1}
Assume PDE \rf{pde-lin} is uniformly parabolic. 
Let the coefficients of $L$ and the function $g$ be of class $\C^{\frac{\beta}2,1+\beta}_b([0,T]\x \Rnu)$, 
$\beta\in (0,1)$.
Further, let the initial condition $\ffi$ be of class $\C^{2+\beta}_b(\Rnu)$.  
Then, the solution $u(t,x)$ of \rf{pde-lin}, whose existence was established by Theorem \ref{fundamental},
belongs to class $\C^{1, 3}_b([0,T]\x\Rnu)$, and its derivative $v(t,x) = \pl_x u(t,x)$
is the unique solution to 
\eq{
\lb{d1}
Lv + \pl_x a \, \pl^2_{xx} u + \pl_x b\, \pl_x u + \pl_x c  \, u + \pl_x g  = 0,\\
v(0,x) = \ffi'(x).
}
In particular, we can differentiate PDE  \rf{pde-lin} w.r.t. $x$.
\end{pro}
\begin{proof}
Introduce the function $u_\Dl(t,x) = \frac{u(t, x+\Dl x) - u(t,x)}{\Dl x}$. Since $u$ is a solution
to \rf{pde-lin}, the linear PDE for $u_\Dl$ takes the form
\mmm{
\lb{pde-dl}
(Lu_\Dl)(t,x) =-\big( \td\pl_x a(t,x) \pl^2_{xx} u(t,x+\Dl x) + \td\pl_x b(t,x) \pl_x u(t,x+\Dl x) \\
+ \td\pl_x c(t,x) u(t,x+\Dl x) + \td\pl_x g(t,x)\big),
}
where $\td\pl_x$ is defined as follows: $\td\pl_x \phi(x) = \int_0^1 \pl_x\phi(x+\la\Dl x) d\la$.
Remark that, by Theorem \ref{hold-unique}, $u$ is of class $\C^{1+\frac{\beta}2, 2+\beta}_b$, 
and, therefore, by assumptions, the right-hand side of \rf{pde-dl} is of class $\C^{0,\beta}_b$.  
By Theorem \ref{fundamental},  the Cauchy problem consisting of PDE \rf{pde-dl} 
and the initial condition $u_\Dl(0,x) = \frac{\ffi(x+\Dl x) - \ffi(x)}{\Dl x}$ has a unique solution which takes the form
\mm{
 u_\Dl(t,x) = \int_{\Rnu}\Gm(t,x,0,z) u_\Dl(0,z)  dz - 
 \int_0^t \int_{\Rnu} \Gm(t,x,s,z) \big(\td\pl_x a(s,z) \pl^2_{xx} u(s,z+\Dl x) \\ 
 + \td\pl_x b(s,z) \pl_x u(s,z+\Dl x) + \td\pl_x c(s,z) u(s,z+\Dl x) + \td\pl_x g(s,z)\big) ds dz.
 }
 On the other hand, consider problem \rf{d1} w.r.t. $v$.
 By Theorem \ref{fundamental}, \rf{d1} has a unique solution $v(t,x)$ which takes the form
 \mm{
  v(t,x) = \int_{\Rnu}\Gm(t,x,0,z) \ffi'(z)  dz - 
 \int_0^t \int_{\Rnu} \Gm(t,x,s,z) \big(\pl_x a(s,z) \pl^2_{xx} u(s,z) \\ 
 + \pl_x b(s,z) \pl_x u(s,z) + \pl_x c(s,z) u(s,z) + \pl_x g(s,z)\big) ds dz.
 }
 Recalling that the fundamental solution $\Gm(t,x,s,z)$ possesses bounds by Gaussian densities \cite{friedman}, we conclude
 that as $\Dl x\to 0$, $u_\Dl(t,x) \to v(t,x)$. This means that
 $v=\pl_x u$. In particular, it means that the derivatives $\pl^3_{xxx} u$ and $\pl^2_{xt} u$ exist,
 and we can differentiate PDE \rf{pde-lin} w.r.t. $x$. 
 \end{proof}
 \begin{lem}
 \lb{ux-bound}
Let (A1)--(A5) hold, and let $u$ be the solution to problem \rf{final}
(whose existence, together with the existence of the bound $M_1$ for its gradient $\pl_x u$,
 was established under (A1)--(A3)).
 Then, there exists a positive function $m(t)$,
 such that one of the alternatives in \rf{D-bound} is fulfilled.
\end{lem}
\begin{proof}
Problem \rf{final} can be rewritten as a linear problem as follows
\eq{
\lb{eq-main}
\frac12 \td\sg^2(t,x) \pl^2_{xx} u + \td f(t,x)\pl_x u + \td g(t,x) + \pl_t u = 0,\\
u(T,x) = h(x),
} 
where   
$\td g(t,x) = g(t,x,u(t,x), \pl_xu(t,x)\sg(t,x,u(t,x)))$, and $\td\sg$ and $\td f$ are defined by \rf{coeff-sde}.
By Proposition \ref{pr1}, we can differentiate PDE
\rf{eq-main} w.r.t. $x$. By doing so, we obtain the following PDE for $v(t,x) = \pl_x \te(t,x)  = \pl_x u(T-t,x)$
\aaa{
\lb{eq-d}
a(t,x,\te) \, \pl^2_{xx}v + b(t,x,\te,\pl_x \te) \pl_x v + c(t,x,\te,\pl_x \te) v - \pl_t v  = - \pl_x g(t,x,\te,\pl_x \te),
}
where  $\te(t,x) = u(T-t,x)$, and the functions $v$, $\te$, $\pl_x v$, and $\pl_x \te$ 
are everywhere evaluated at $(t,x)$. Furthermore, $a$, $b$, and $c$ are defined as follows 
\eq{
\lb{abc}
a(t,\ldots) = \frac12\sg^2(T-t,\ldots); \qquad
b(t,\ldots) = \big(\pl_x a + \pl_p f\sg \pl_x u + \pl_p g\, \sg + f\big)(T-t,\ldots); \\
c(t,\ldots) = \big(\pl_x f + \pl_u g + \pl_p g \pl_x\td\sg + \pl_u f \pl_x u + \pl_p f \pl_x \td\sg \pl_x u\big)(T-t,\ldots),
}
where the dots are used to simplify notation and are to be substituted with $x,\te(t,x), \pl_x\te(t,x)$.
Let $\mc L$ be the partial
differential operator defined by the left-hand side of \rf{eq-d}, i.e., 
\aa{
\mc L v = a\, \pl^2_{xx} v + b\, \pl_x v + c \,v - \pl_t v.
}
If (A5)-(a) is in force, define the function
$\td v(t,x) = v(t,x) - m(t)$, where $m(t) = \int_0^t \td m(s) ds$ and $\td m(s)$ is a positive sufficiently 
small function whose choice is explained below. Then,
\aa{
\mc L \td v = - \pl_x g - c\, m(t) + \td m(t) \lt -\inf_{(x,u,p)\in \mc R} \pl_x g(t,x,u,p) - c\, m(t) + \td m(t).
}
Remark that by (A4), $c$ is bounded.
Therefore, if $\td m(t)$ and $m(t)$ are sufficiently small, then $\mc L \td v \lt 0$. 
Further, since $m(0) = 0$, then $\td v(0,x) \gt 0$.
By Theorem \ref{comparison}, $\td v(t,x) \gt 0$, and, therefore $v(t,x) \gt m(t)$ on $[0,T]\x \Rnu$. 
If (A5)-(b) is in force, then, defining the function $\td v(t,x) = v(t,x) + m(t)$, we obtain 
that $\mc L \td v = - \pl_x g + c m(t) - \td m(t)$. By a similar argument, we conclude that $v(t,x) \lt -m(t)$ on 
$[0,T]\x \Rnu$. The lemma is proved.
\end{proof}
As a corollary of Theorem \ref{privault-dung} and Lemma \ref{ux-bound}, we obtain  Gaussian
estimates for the density of the law of $Y_t$.
\begin{thm}
\lb{y-est}
Let (A1)--(A5) hold. Then, the distribution of $Y_t$ has a density $p_{Y_t}$ with respect to the
Lebesgue measure. Moreover, for almost all $x\in \Rnu$, this density satisfies the estimate
\aaa{  
  \lb{est1} 
\frac{\E |Y_t-\E[Y_t]|}{2\La(t)} \exp\Big(-\frac{\big(x-\E[Y_t]\big)^2}{2\la(t)}\Big) \lt 
p_{Y_t}(x) \lt \frac{\E |Y_t-\E[Y_t]|}{2\la(t)} \exp\Big(-\frac{\big(x-\E[Y_t]\big)^2}{2\La(t)}\Big),
}
where $\la(t)$ and $\La(t)$ are positive functions that can be computed explicitly. 
Further, for all $x>0$, the tail probabilities of $Y_t$ satisfy
\aaa{
\lb{tail}
\PP(Y_t > x) \lt \exp\Big(-\frac{\big(x-\E[Y_t]\big)^2}{2\La(t)}\Big) \quad \text{and} \quad
\PP(Y_t<-x) \lt \exp\Big(-\frac{\big(x+\E[Y_t]\big)^2}{2\La(t)}\Big).
}
\end{thm}
\begin{proof}
Since $D_rY_t = \pl_x u(t,X_t)D_rX_t$, by \rf{dx-est} and Lemma \ref{ux-bound},
 \eq{
 \lb{Y-bounds}
\text{either \; } &m(t)\nu(M) e^{-M_\psi t} \lt  D_r Y_t \lt M_1 \mu(M) e^{M_\psi t} \; \text{\; a.s.} \\
\text{or \; } &m(t)\nu(M) e^{-M_\psi t} \lt - D_r Y_t \lt M_1 \mu(M) e^{M_\psi t} \; \text{\; a.s.,} 
  }
  where $M_1$ is the bound for $\pl_x u$.
Taking into account that $D_rX_t = 0$ if $r>t$, we obtain
\aa{
\la(t) = t\big( m(t)\nu(M) e^{-M_\psi t} \big)^2 \lt \int_0^t D_r Y_t \,\E[D_rY_t|\mc F_r] dr < t\big(M_1\mu(M) e^{M_\psi t}\big)^2 = \La(t).
}
By Theorem \ref{privault-dung}, $Y_t$ has a density with respect to the Lebesgue measure,
and estimate \rf{est1} holds. Also,  we have estimates for the tail probabilities of $Y_t$,
given by \rf{tail}. 
\end{proof}
\subsection{Density estimates for the law of $Z_t$}
To estimate the density for $Z_t$, we recall that $Z_t = \pl_x u(t,X_t) \sg(t,X_t,u(t,X_t))$.
This immediately implies that $Z_t$ is Malliavin differentiable, and 
\aaa{
\lb{dz}
D_r Z_t = \big(\pl_x u(t,X_t)  \pl_x \td\sg(t,X_t) 
+ \pl^2_{xx}u(t,X_t) \td \sg(t,X_t)\big)D_rX_t,
}
where $\td\sg(t,x) = \sg(t,x,u(t,x))$.
Lemma \ref{lem2} below provides a lower bound for the derivative $\pl^2_{xx} u$.
\begin{lem}
\lb{lem2}
Let (A1)--(A4), (A5'), and (A7)--(A9) hold, and let $u$ be the solution to problem \rf{final}. 
Then there exists a positive function $\rho(t)$ such that
$\pl^2_{xx} u \gt \rho(t)$ for all $(t,x) \in [0,T]\x\Rnu$.
\end{lem}
\begin{proof}
Remark that linear PDE \rf{eq-d} takes form \rf{pde-lin} with $a$, $b$, and $c$ given by \rf{abc}.
Since $\te(t,x) = u(T-t,x)$
is of class $\C^{1+\frac{\beta}2, 2+\beta}_b([0,T]\x\Rnu)$, then, by (A7), the coefficients
$a(t,x,\te(t,x), \pl_x\te(t,x))$, $b(t,x,\te(t,x), \pl_x\te(t,x))$, and $c(t,x,\te(t,x), \pl_x\te(t,x))$ of PDE \rf{eq-d} 
and its right-hand side $-\pl_x g(t,x,\te(t,x),\pl_x \te(t,x))$
are of class $\C^{\frac\beta2,1+\beta}_b$ as functions of $(t,x)$. By Proposition \ref{pr1}, the solution $v=\pl_x u$
to \rf{eq-d} is of class $\C^{1,3}_b$ (and,  therefore, $u$ is of class $\C^{1,4}_b$), 
and we can differentiate PDE \rf{eq-d} w.r.t. $x$. 

Defining $w=\pl^2_{xx} u$ and replacing  $\pl^2_{xx} u$, $\pl^3_{xxx} u$, and $\pl^4_{xxxx}u$
by $w$, $\pl_x w$, and $\pl^2_{xx} w$, respectively, everywhere where it is possible, we obtain the following
PDE w.r.t. $w$
\mmm{
\lb{last2}
 a \, \pl^2_{xx} w + b\, \pl_x w + \mc P w -\pl_t w \\ 
=   - \pl^2_{xx} g - \Psi_1 \pl_x u - \Psi_2 (\pl_x u)^2 - \Psi_3 (\pl_x u)^3
- \Psi_4 (\pl_x u)^4 - \Psi_5 (\pl_x u)^5,
}
where $\mc P$ is a polynomial of  $\sg$, $f$, $g$, all their first 
and second order derivatives w.r.t. $x$, $u$, $p$, and, additionally, of $\pl_x u$.
Further, the functions
$\Psi_i$, $i=1,2,3,4,5$, are defined by the right-hand sides of the first, second, third, fourth, and the fifth 
inequalities, respectively, in assumption (A9). Let
$\mc L_1$ denote the partial differential operator defined by the left-hand side of \rf{last2}.
We proceed with the same argument as in
Lemma \ref{ux-bound}, that is, define the function
$\td w(t,x) = w(t,x) - \rho(t)$, where $\rho(t) = \int_0^t \td \rho(s) ds$ and $\td \rho(s)$ is a sufficiently small positive function. Then,
\aa{
\mc L_1 \td w = - \pl^2_{xx} g - \sum_{n=1}^5 \Psi_n (\pl_x u)^n  - \mc P\, \rho(t) + \td \rho(t).
}
Remark that  under (A5'), $\pl_x u\gt 0$ on $[0,T]\x \Rnu$. Indeed, this follows from the proof of Lemma \ref{ux-bound},
where we have to apply Theorem \ref{comparison} with $m(t) = 0$. 
Hence, by (A9),  $\sum_{n=1}^5 \Psi_n (\pl_x u)^n\gt 0$. Further, by (A2)--(A4), $\mc P$ is bounded.
Therefore, (A7) implies that if $\td \rho(t)$ and $\rho(t)$ are sufficiently small, then $\mc L_1 \td w \lt 0$. Since $h''\gt 0$,
by Theorem \ref{comparison}, we obtain that $w(t,x) \gt \rho(t)$ for all $(t,x)\in [0,T]\x\Rnu$.
\end{proof}
\begin{thm}
Let (A1)--(A4), (A5'), and (A6)--(A9) hold. Then, the distribution of $Z_t$ has a density $p_{Z_t}$ with respect to the
Lebesgue measure. Moreover,  for almost all $x\in \Rnu$, this density satisfies 
\aaa{  
  \lb{est2} 
\frac{\E |Z_t-\E[Z_t]|}{2\Sigma(t)} \exp\Big(-\frac{\big(x-\E[Z_t]\big)^2}{2\sig(t)}\Big) \lt 
p_{Z_t}(x) \lt \frac{\E|Z_t-\E[Z_t]|}{2\sig(t)} \exp\Big(-\frac{\big(x-\E[Z_t]\big)^2}{2\Sigma(t)}\Big),
}
where $\sig(t)$ and $\Sigma(t)$ are positive functions that can be computed explicitly. 
Further, for all $x>0$, the tail probabilities of $Z_t$ satisfy
\aaa{
\lb{tail2}
\PP(Z_t > x)  \lt \exp\Big(-\frac{\big(x-\E[Z_t]\big)^2}{2\Sigma(t)}\Big)
\quad \text{and} \quad 
\PP(Z_t<-x)  \lt \exp\Big(-\frac{\big(x+\E[Z_t]\big)^2}{2\Sigma(t)}\Big).
}
\end{thm}
\begin{proof}
Assumptions (A5') and (A6)--(A9) provide the lower bound for the function
\aaa{
\lb{func}
\pl_x u \, \pl_x \sg + (\pl_x u)^2\pl_u \sg + \pl^2_{xx}u\,  \sg
}
on the right-hand side of \rf{dz}. Indeed,  $\pl_x u \, \pl_x \sg + (\pl_x u)^2\pl_u \sg \gt 0$ by (A5') and (A6). Finally, from (A1) and (A7)--(A9), by virtue of Lemma \ref{lem2}, it follows that
 $\pl^2_{xx}u\, \sg \gt \rho(t) \nu(M)$, where $\rho(t)$ is the positive function defined in Lemma \ref{lem2}
and $\nu(\fdot)$ is the function from (A1). Now taking into account that $D_rX_t$ possesses
upper and lower positive bounds, provided by
\rf{dx-est}, we obtain that the Malliavin derivative 
$D_r Z_t$ satisfies
\aa{
\nu(M)^2 \rho(t) e^{-M_\psi t}\lt D_rZ_t \lt \mu(M) \gm \,e^{M_\psi t} \quad \text{a.s.},
}
where $\gm$ is an upper bound for \rf{func}. This bound, indeed, exists by (A3) and since 
the solution $u$ has bounded derivatives. Now by the same argument as in Theorem \ref{y-est},
we obtain that the law of $Z_t$ possesses a density $p_{Z_t}$ w.r.t. the Lebesgue measure, and
\rf{est2} holds with $\Sigma_t = t \mu(M)^2 \gm^2 e^{2 M_\psi t}$ and 
$\sig(t) = t \nu(M)^4 \rho(t)^2 e^{-2 M_\psi t}$. Moreover, we obtain estimates \rf{tail2} for the
tail probabilities of $Z_t$.
\end{proof}
\section*{Acknowledgements}

Christian Olivera  is partially supported by FAPESP 	by the grants 2017/17670-0 and 2015/07278-0.

\end{document}